\documentclass[11pt]{amsart}
\usepackage{comment}

\usepackage[pdfencoding=auto, psdextra, hidelinks, pdfusetitle]{hyperref}
\usepackage{preamble/packages}
\usepackage{preamble/environments}
\usepackage{preamble/notation}
\usepackage[margin=2.9cm]{geometry}

\graphicspath{{./image/}}
\author{Vasily Krylov, Raphaël Paegelow, Pavel Shlykov}
\title{K-theory of the Gieseker variety and type A cyclotomic Hecke algebra}

\begin{document}

\maketitle
\setcounter{tocdepth}{1}
  \tableofcontents

  \section{Introduction}\label{sec:Intro}
  


\subsection{ADHM spaces and Gieseker varieties} So-called ADHM spaces \cite{ADHM} form a fundamental class of algebraic varieties that were used to describe framed $U(r)$-instantons on $S^4=\mathbb{R}^4 \cup \{\infty\}$ by means of linear algebra. 
The ADHM space has a natural partial compactification which is a non-singular algebraic variety.  
Following \cite{Nek03}, we will call it the \textit{Gieseker space} or the \textit{Gieseker variety} and denote it by $\Gis$. It also admits a description as a moduli space of framed torsion-free sheaves on $\BP^2(\BC)$ of rank $r$ and  second Chern class $n$ \cite{Nakajima-Yoshioka-instanton-counting}. This variety has been used extensively by physicists. For instance, \cite{Nek03} defined instanton partition functions as equivariant localizations on the Gieseker space. Nakajima quiver varieties \cite{Nak1994} are a natural generalization of these spaces.

\subsection{Relation to Higgs/Coulomb branches and symplectic duality} The Gieseker space is a symplectic resolution of the singularities of the {\emph{Higgs}} branch of the $3d$ $\mathcal{N}=4$ quiver gauge theory corresponding to the Jordan quiver.
In \cite{KS25}, the first and third authors proposed a general approach that allowed them to identify equivariant cohomology of the Gieseker space with the algebra of functions on the schematic $\mathbb{C}^\times$-fixed points of the corresponding BFN Coulomb branch (\cite{BFNII}), thus proving the non-quantized equivariant Hikita conjecture relating Higgs and Coulomb branches (see, for example, \cite[Section 5.6]{kamnitzer_symplectic-duality}, where the quantized version is formulated). In \cite{KS25}, the first and third authors used a similar approach to identify the equivariant cohomology of the Gieseker space with the center of the degenerate cyclotomic Hecke algebra, resulting in a new simple proof of the main result of \cite{SVV}, see also \cite{Web15}.  

\subsection{Generalization to $K$-theory} It is natural to replace the equivariant cohomology of the Gieseker space by the equivariant $K$-theory.
A $K$-theoretic version of the Hikita conjecture was formulated in \cite{dumanski-krylov_2025} (see also \cite[Appendix B]{Zhou23-Virtual-Coulomb-branch}) and one could ask whether it holds for the Gieseker space and whether the equivariant $K$-theory of this space allows for a description in terms of cyclotomic Hecke algebras, similar to those of \cite{SVV}, \cite{Web15} and \cite{KS25}.

\subsection{Main result} 
The main goal of this paper is to apply the approach of \cite{KS25} to obtain  an algebraic description of the equivariant $K$-theory of the Gieseker space (see \cref{thm:main theorem formulated in introduction} below) which is precisely the $K$-theoretic analog of the main theorem of \cite{SVV}. 

\vspace*{0.3cm}

The type $A$ Hecke algebra is a deformation of the group algebra of the symmetric group. The cyclotomic Hecke algebra is a quotient of the affine Hecke algebra of the symmetric group. Cyclotomic Hecke algebras were introduced for all complex reflection groups by Broué and Malle \cite{BM93}. At the same time, Ariki and Koike defined a Hecke algebra for the complex reflection group $G(r,1,n)$ that  generalizes the Iwahori-Hecke algebras of types A and B \cite{ArikiKoike}. This algebra turned out to be the cyclotomic Hecke algebra for affine type A.  
In \cite{Pae26}, the second author described combinatorial correspondences between a fixed point locus of the Gieseker space and blocks of a certain specialized Ariki-Koike algebra.

Here we propose a common framework to generalize these observations. 
The main theorem of this paper is:
\begin{customthm}{A}[\cref{thm:main}]\label{thm:main theorem formulated in introduction}
The equivariant $K$-theory of the Gieseker space is isomorphic to the Jucys-Murphy center of the cyclotomic Hecke algebra, as an algebra over the equivariant $K$-theory of a point. 
\end{customthm}
In particular, we obtain a basis in the Jucys-Murphy center coming from the geometry (see  \cref{rem:basis in JM center}).

\subsection{Relation to the center conjecture}
The following conjecture was formulated in \cite[Conjecture 3.1]{McG12} and is proved by Hu, Mathas, Shi, and Webster in \cite[Theorem B]{HMSW26}.


\begin{conj}\label{conj in the introduction: JM ceneter is center}
For $q \neq 1$, the Jucys-Murphy center of the cyclotomic Hecke algebra coincides with its center. 
\end{conj}

In \cite[Lemma 7.1.7]{GeckPeiffer} (see also \cite[Equation (2.5)]{HuShi}) it's proven that the center of the cyclotomic Hecke algebra is {\emph{dual}} to the cocenter under a certain natural pairing. Note that the cocenter is a ``simpler'' object than the JM-center from the algebraic point of view as, for example, it is known to have an explicit basis (see \cite[Theorem 1.4(1)]{HuShi}). Note now that the equivariant $K$-theory of the Gieseker space pairs perfectly with the equivariant $K$-theory of a certain closed projective subvariety of the Gieseker space that we will denote by $\mathcal{L}(n,r)$ (this is a general feature of Nakajima quiver varieties, see \cite[Theorem 7.3.5]{NakJAMS}). So, Theorem \ref{thm:main theorem formulated in introduction} implies that the Jucys-Murphy center is {\emph{dual}} to the equivariant $K$-theory of $\mathcal{L}(n,r)$. We conclude that Conjecture \ref{conj in the introduction: JM ceneter is center} is equivalent to the following conjecture (compare with \cite[Theorem 3.37]{SVV}). See Section \ref{sec:center-cocenter} for more details.

\begin{conj}\label{conj in the introduction: K theory of L is cocenter}
For $q\neq 1$, there exists an isomorphism between the  equivariant $K$-theory of $\mathcal{L}(n,r)$ and  the cocenter of the cyclotomic Hecke algebra compatible with the natural pairings. 
\end{conj}

\subsection{Corollaries of the main result}
In \cite{Pae26}, the second author identifies connected components of fixed-point loci in Gieseker space at roots of unity (also described as affine type $A$ quiver varieties)  with blocks of the cyclotomic Hecke algebra (at the same roots of unity). After specializing the parameters to roots of unity, we naturally recover the aforementioned bijection observed in \cite{Pae26} as a corollary of Theorem \ref{thm:main theorem formulated in introduction}.
Moreover, up to \cref{conj in the introduction: JM ceneter is center}, we obtain the following upgrade of \cite{Pae26}.  


\begin{customthm}{B}[\cref{thm:blocks}]
 The $K$-theory of an affine type A quiver variety 
 is isomorphic, as a $\C$-algebra, to the center of the corresponding block of the specialized cyclotomic Hecke algebra.    
\end{customthm}

\begin{rmk}
A related result with $K$-theory being replaced by cohomology was obtained by Webster (also assuming \cref{conj in the introduction: JM ceneter is center}) in \cite[Proposition 3.9]{Web15} using different methods. Webster identifies (equivariant) cohomology of affine type $A$ quiver varieties with centers of certain KLR algebras. 
\end{rmk}


For $r=1$, the Gieseker space $\mathcal{G}(n,r)$ becomes isomorphic to the Hilbert scheme $\operatorname{Hilb}_n(\mathbb{A}^2)$ of $n$ points on $\mathbb{A}^2$. In this case, the affine cyclotomic Hecke algebra is isomorphic to the Hecke algebra of type $A$. Moreover, by \cite[Theorem 7.2]{FG06} for $q \neq 1$ (i.e., after inverting $q-1$) the Jucys-Murphy’s center of the type $A$ Hecke algebra identifies with its center. 
So, Theorem \ref{thm:main theorem formulated in introduction} in this case gives the following. 
\begin{specialcorollary}{C}[\cref{cor:Hilb case}]
The equivariant $K$-theory of the Hilbert scheme $\operatorname{Hilb}_n(\mathbb{A}^2)$ is isomorphic to the Jucys-Murphy’s-center of the type $A$ Hecke algebra. For $q \neq 1$, this is also the center. 
\end{specialcorollary}

\begin{warning}
What we call Jucys-Murphy’s center of the cyclotomic type $A$ Hecke algebra does not generalize Jucys-Murphy center in \cite[Theorem 7.2]{FG06}. Indeed, Jucys-Murphy elements are related by the following relation: $\Lj_i=(q-1)L_i+1$ where $\Lj_i$ are the Jucys-Murphy element that we consider whereas $L_i$ are the Jucys-Murphy element considered in \cite{FG06}. The two Jucys-Murphy center coincide for $q \neq 1$ but their  specializations to $q=1$ are different.
\end{warning}

\subsection{Structure of the paper}
This paper is structured as follows.
In \cref{sec:Preliminaries} we recall the definition of the Gieseker space and some basic facts about it. In \cref{sec:K-theory_of_Gieseker_space} 
we prove \cref{thm:main theorem formulated in introduction}. 
In \cref{sec:center-cocenter} we 
reformulate \cref{conj in the introduction: JM ceneter is center} in terms of the identification between the cocenter $\operatorname{Tr}(\Heck_{n,r})$ of $\Heck_{n,r}$ and the equivariant $K$-theory of $\mathcal{L}(n,r)$   and prove that the former embeds into the latter.
Finally, 
 in \cref{sec:Applications} we obtain various corollaries of \cref{thm:main theorem formulated in introduction}. Namely, in the  special case of $r=1$, we identify the equivariant $K$-theory of $\operatorname{Hilb}_n(\mathbb{A}^2)$ with the Jucys-Murphy’s  center of the type $A$ Hecke algebra. For general $r$ we then proceed to study the special cases $q=1$ and $q$ being a root of unity. For the latter we prove \cref{thm:blocks}.

 \subsection*{AI declaration}
 ChatGPT $6.1$ Sol Ultra was used for proofreading.  The authors bear full responsibility for the contents of this paper. 
 
\subsection*{Acknowledgements}
We are grateful to Alexander Braverman, who drew the attention of the first and third authors to the question of describing the equivariant $K$-theory of a Hilbert scheme algebraically. The first author was supported by the Simons Foundation Award 888988 as part of the Simons Collaboration on Global Categorical Symmetries. The work was supported by the National Science Foundation under Grant No. DMS-
2424139 while the first author was in residence at the Simons Laufer Mathematical Sciences Institute in Berkeley,
California, during the Fall 2026 semester.
The second author is supported by the SFB-TRR 195 ``Symbolic Tools in Mathematics and their Application” of the German Research Foundation (DFG). We are also grateful to Jun Hu and Lei Shi for the useful correspondences on the draft of the paper. The third author was supported by the Engineering and Physical Sciences Research Council [grant number EP/W013053/1].  
We are grateful to Jun Hu, Andrew Mathas, Lei Shi, and Ben Webster for informing us about their ongoing work.


  \section{Preliminaries}\label{sec:Preliminaries}
  By an algebraic variety, we mean an reduced separated scheme, with a structure morphism that is assumed to be separated and of finite type over $\C$. 


\subsection{Gieseker space}

For $n,r \in \mathbb{Z}_{\geqslant 1}$, let $V:= \BC^n, W:=\BC^r$ be two vector spaces. Denote by $\repspace$ the space
$  \repspace:=\Hom(V,V) \oplus \Hom(W,V)$ 
and by $\cotrepspace$ its cotangent bundle:
\begin{equation}\label{eqn:cotrepspace}
    \cotrepspace:=\Hom(V,V) \oplus \Hom(V,V) \oplus \Hom(W,V) \oplus \Hom(V,W),
\end{equation}
which is the representation space of the doubled framed Jordan quiver. 

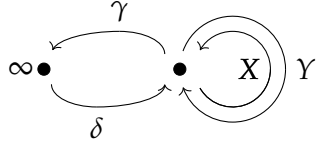
\begin{figure}[h]
  \centering

  \begin{tikzpicture}
    \draw[fill = black] (-0.7 ,0) circle (0.08);
    \draw[fill = black] (-2.5 ,0) circle (0.08) node [anchor = east] {$\infty$};
    \draw ([shift=(160 : 0.7)]0, 0) arc (160 : 0: 0.7) node [anchor = west] {$Y$};
    \draw ([shift=(-150 : 0.5)]0, 0) arc (-150 : 0: 0.5) node [anchor = east] {$X$};
    \draw (-0.9, 0.2) arc (0 : 90: 0.6 and 0.3 ) node [anchor = south] {$\delta$};
    \draw (-2.4, -0.2) arc (-180 : -90: 0.6 and 0.3 ) node [anchor = north] {$\gamma$};
    \draw[->] ([shift=(90:0.5)]-1.5, 0)  arc (90 : 170 :  0.9 and 0.3);
    \draw[->] ([shift=(-90:0.5)]-1.8, 0)  arc (-90 : 0 :  0.9 and 0.3);
    \draw ([shift=(-150 : 0.5)]0, 0) arc (-150 : 0: 0.5) node [anchor = east] {$X$};
    \draw[->] (0.5, 0)  arc (0 : 150 :  0.5);
    \draw[->] (0.7, 0)  arc (0 : -160 :  0.7);
  \end{tikzpicture}
    \caption{Double framed Jordan quiver}
  \label{double jordan quiver}
\end{figure}

We can represent the elements of $\cotrepspace$ as quadruples $(X,Y, \gamma, \delta)$ where each of $X,Y, \gamma, \delta$ belong to the corresponding direct summands in \cref{eqn:cotrepspace}. The space $\cotrepspace=T^*\repspace$ carries a natural symplectic form. The group $G_{\vv}:=\GL(V)$ acts naturally on $\cotrepspace$. The action is Hamiltonian with moment map $\mu$ 
\begin{equation*}
\mu\colon \cotrepspace \ra \mathfrak{gl}_{V},\, (X,Y,\gamma,\delta) \mapsto [X,Y]+\gamma\delta
\end{equation*}
where we identify $\mathfrak{gl}_{V}$ with its dual thanks to the Killing form.
\begin{dfn}
A quadruple $(X,Y,\gamma,\delta) \in \cotrepspace$ is called {\em{stable}} if for every $X, Y$-invariant subspace $S \subset V$ such that $S$ contains $\on{Im}(\gamma)$ we have $S=V$. 
We denote by $\cotrepspace^{\mathrm{st}} \subset \cotrepspace$ the (open) subset of stable quadruples.
\end{dfn}
The action of $G_{\vv}$ on the space $\cotrepspace^{\mathrm{st}}$ is free.
\begin{dfn}
The Gieseker variety $\Gis$ is then defined as the following quotient: 
\begin{equation*}
\Gis:=\mu^{-1}(0)^{\mathrm{st}}/G_{\vv}.
\end{equation*} 
\end{dfn}

For $r=1$ the variety $\Gie(n,1)$ is isomorphic to $\on{Hilb}_{n}(\BA^{2})$, the Hilbert scheme of $n$ points on $\BC^2$ (\cite[Theorem 2.1, Proposition 2.10]{NakajimaBook}).  

The Gieseker variety $\Gis$ is naturally equipped with the action of the group $\GL_2\times \GL_r$ -- it descends from the action on $\cotrepspace$ as it commutes there with the action of $G_{\vv}$.

\begin{equation*}
(g,t) \cdot (X,Y,\gamma,\delta):=(aX + b Y, cX + d Y, \gamma t^{-1}, \mathrm{det}(g) t \delta)
\end{equation*}
where $g=\begin{pmatrix} a & b \\ c & d \end{pmatrix} \in \GL_2$ and $t \in \GL_r$. This action commutes with the $G_{\vv}$-action defined above and thus induces an action on $\Gie(n,r)$.

In particular, we are interested in the torus part of this action. Let $\BC^\times_{\hbar}$ denote the scalar diagonal torus of $\GL_2$, $\mathbb{T}$  the maximal diagonal torus of $\SL_2$ and $\BT_r$ the maximal torus of $\GL_r$. Denote $\torus := \mathbb{T}\times \BT_r$ and let $\CR:=K_{\torus}(\{*\})$, which is the ring $\C[q^{\pm 1}, Q_1^{\pm 1},\dots,Q_r^{\pm 1}]$, where we normalize the character $q(\mathrm{diag}(t,t^{-1})):=t$.
Notice that $\BC^\times_{\hbar}$ scales the symplectic form while $\mathbb{T}$ and $\mathbb{T}_r$ act symplectically. Also note that we consider all our $K$-groups to be tensored with $\BC$.

The following is a short summary of the relevant parts of  \cite[Section 4]{KS25}.
First of all, recall the classical fact  that the set $\on{Hilb}_{n}(\BA^{2})^{\torus}$ is identified with the set of partitions $\lb$ of $n$. In particular, each ideal corresponding to a fixed point is monomial in $x,y$ -- coordinates on $\BC^2$. The identification above associates to each ideal the corresponding Young diagram of $\lb$
\begin{equation*}
\BY(\la)=\{(i,j)\, | \, 1 \leqslant i \leqslant l,\, 1 \leqslant j  \leqslant \lb_i\},
\end{equation*}
where $l$ is the length of $\lambda$.
We fill $\BY(\la)$ with monomials by putting $x^{i-1}y^{j-1}$ into the box $(i,j) \in \BY(\la)$.
	The ideal $J_{\la}$ that corresponds to $\la$ is spanned by the monomials outside of the diagram.

\begin{deff}\label{def:content}
The content of a node $(a,b)$  of the Young diagram $\BY(\la)$ is denoted by $\ct(a,b)$ and is equal to the integer $b-a$.
\end{deff}

 For each $\torus$-fixed $J_\lb$, the quotient $\mathbb{C}[x,y]/J_\lb$ has a basis consisting of classes of $(x^{i-1}y^{j-1})_{(i,j) \in \BY(\lb)}$. On the vector $x^{i-1}y^{j-1}$ the torus $\BT$ acts by $q^{\ct(i,j)}$. The torus $\BT_1$ acts trivially on the Hilbert scheme. 

Next, consider the $\Gis$ for $r$ not necessarily equal to 1. It is known \cite[page 18]{Nakajima-Yoshioka-instanton-counting} that the variety $\Gis^{\BT_r}$ is 
isomorphic to the disjoint union 
\begin{equation} \label{eqn:gieseker_fixed_points_into_hilb}
	\bigsqcup\limits_{\sum\limits_{j=1}^{r} n_{j}=n}~\prod\limits_{j=1}^{r} \Gie({n_{j}},1)~ \simeq 
	\bigsqcup\limits_{\sum\limits_{j = 1}^{r} n_{j}=n}~\prod\limits_{j =1}^{r} \on{Hilb}_{n_j}(\BC^2).
	\end{equation}

Denote by $\mpart$ the set of $r$-multipartitions of $n$. The identification above implies the following proposition (the idea of this parametrisation of the torus fixed points goes back to \cite[Lemma 3]{VVcyclic}):

\begin{prop} \label{prop:gieseker_fixed_point_structure}
	The set of fixed points $\Gis^{\torus}$
	is identified with the set $\mpart$. A multipartition $\lbb=(\la^1,\la^2,\ldots,\la^{r})$ corresponds to the following quiver data
	$ X^{\lbb},Y^{\lbb} \in \on{End}(V),\,
	 \newline \gamma^{\lbb} \in \on{Hom}(W,V),\, 
  \delta^{\lbb} \in \on{Hom}(V,W)$:
	\begin{equation*} 
	V=\bigoplus_{j=1}^{r} \BC[x,y]/J_{\la^j}, \qquad
	W= \bigoplus_{j=1}^{r} \BC w_j;
	\end{equation*}
	\begin{equation*}
	X^{\lbb} := \bigoplus_{j=1}^{r} L^{j}_{x}, \qquad
	Y^{\lbb} := \bigoplus_{j=1}^{r} L^{j}_{y};
	\end{equation*}
	\begin{equation*}
	\gamma^{\lbb} \hspace{0,1cm} \on{sends} \hspace{0,1cm} w_{j} \in W \hspace{0,1cm} \on{to} \hspace{0,1cm} 1 \hspace{0,1cm} \on{in} \hspace{0,1cm} \BC[x,y]/J_{\la^j},~
	\delta^{\lbb} := 0,
	\end{equation*}
where  $L^{j}_{x},\,L^{j}_{y}$  denote the operators of multiplications by $x,y$ on $\BC[x,y]/J_{\la^j}$.
 
	\end{prop}

The following is a rephrasing of \cite[Lemma 4.7]{KS25}.
\begin{prop} \label{prop:weights_at_fixed_point}
     The torus $\torus$ acts on the vector $[x^{i-1}y^{j-1}] \in \BC[x,y]/J_{\lb^k}$ in the fiber of the tautological bundle over the fixed point $\lbb$ by $q^{\ct(i,j)}Q_k$.  
\end{prop}
    

  \section{Cyclotomic affine Hecke algebra and K-theory of Gieseker space}\label{sec:K-theory_of_Gieseker_space}
  \subsection{Cyclotomic affine Hecke algebra}
Let $R$ be a commutative ring with unit $1$. Let $q$ be an invertible element of $R$ and $Q_1,\dots, Q_r$ in $R$ and denote by $\boldsymbol{q}=(q,Q_1,\dots,Q_r)$. Let us start by defining $\Heck_{n,r}(R,\boldsymbol{q})$, the cyclotomic Hecke algebra of the reflection group 
\begin{equation*}
\mathrm{G}(r,1,n):=(\Z/r\Z)^n \rtimes \mathfrak{S}_n
\end{equation*}
over $R$ with Hecke parameter $q$ and cyclotomic parameters $Q_i$. This is the Ariki-Koike algebra over $R$. We recall \cite[Definition 3.1]{ArikiKoike}.

\begin{deff}\label{deff:cyclotomic Hecke}
Let $\Heck_{n,r}(R,\boldsymbol{q})$ be the unital, associative $R$-algebra with generators $\Lj_1, \ldots, \Lj_n, \Tj_1,\dots,\Tj_{n\smin 1}$ and relations:
\begin{equation*}
\left\{
\begin{aligned}
\Tj_i\Tj_j&=\Tj_j\Tj_i &\quad|i\smin j|>1, \forall (i,j)\in \llbracket 1, n \smin 1 \rrbracket^2,&\\
\Lj_i\Lj_j&=\Lj_j\Lj_i &\quad \forall (i,j)\in \llbracket 1,n \rrbracket^2,&\\
\Tj_i\Lj_j&=\Lj_j\Tj_i &\quad j\neq i, j\neq i+1, \forall (i,j)\in \llbracket 1, n \smin 1 \rrbracket \times \llbracket 1, n \rrbracket,&\\
q\Lj_{i+1}&= \Tj_i\Lj_i\Tj_i  &\forall i\in \llbracket 1, n \smin 1 \rrbracket,&\\
(\Tj_i+1)(\Tj_i-q)&=0 &\forall i\in \llbracket 1, n \smin 1\rrbracket,&\\
\Tj_i\Tj_{i+1}\Tj_i&=\Tj_{i+1}\Tj_i\Tj_{i+1}  &\forall i \in \llbracket 1, n\smin 2 \rrbracket,& \label{1}\\
\prod_{i=1}^r{(\Lj_1-Q_i)}&=0. &&
\end{aligned}\right.
\end{equation*}

\end{deff}

\noindent Since the braid relations hold for $\Tj_1,\dots,\Tj_{n\smin 1}$ in $\Heck_{n,r}(R,\bq)$,  $\Tj_{\omega}:=\Tj_{i_1}\dots \Tj_{i_k}$ does not depend on the choice of a reduced expression of $\omega=s_{i_1}\dots s_{i_k} \in \mathfrak{S}_n$, where $(s_i)_{i \in \llbracket 1, n\smin 1 \rrbracket}$ is the standard set of simple reflections of the symmetric group on $n$ letters. By \cite[Corollary 3.13]{ArikiKoike}, we have that $\Heck_{n,r}(R,\bq)$ is a free $R$-module with basis
\begin{equation}\label{eqn:basis-cycl}
\left \{\Lj_1^{a_1}\dots \Lj_n^{a_n}\Tj_{\omega} | (a_1,\dots,a_n) \in \llbracket 0, r \smin 1 \rrbracket^n, \omega \in \mathfrak{S}_n \right\}.
\end{equation}

\noindent We are interested in $Z(\Heck_{n,r}(R,\bq))$, the center of the cyclotomic Hecke algebra. Let $Z(\Heck_{n,r}(R,\bq))^{\mathrm{JM}}$ be the ring of symmetric polynomials in the Jucys-Murphy elements $\Lj_1,\dots,\Lj_n$. We refer to this subalgebra as the Jucys-Murphy center. We then have the following lemma.

\begin{lemme}
\label{lemma:incl-center}
The ring $Z(\Heck_{n,r}(R,\bq))^{\mathrm{JM}}$ is central in $\Heck_{n,r}(R,\bq)$. 
\end{lemme}
\begin{proof}
The result follows from \cite[Proposition 2.1]{Mat04}.
\end{proof}

For a variety $X$ equipped with an action of a group $H$ we set:
\begin{equation*}
K^H(X) := K_0(\operatorname{Coh}^H X),~K_H(X) := K_0(\operatorname{Vect}^H X).
\end{equation*}
Note that $K_H(X)$ has an algebra structure with a product given by $\otimes$ and $K^H(X)$ is a module over $K_H(X)$.
Recall that $\mathcal{R}=\C[q^{\pm 1}, Q_1^{\pm 1},\dots,Q_r^{\pm 1}]$. By definition, $K_{\GL_n(\C) \times \mathcal{T}}(\{ *\})$ is the representation ring of $\GL_n(\C) \times \mathcal{T}$ which is isomorphic to $\mathcal{R}[x_1^{\pm 1},\dots,x_n^{\pm 1}]^{\mathfrak{S}_n}$. 

\begin{prop} \label{prop:generators_in_k-theory}
The classes $[\C^n],[\Lambda^2  \C^n],...,[\Lambda^n\C^n], [(\Lambda^n\C^n)^{*}] $ generate   $K_{\operatorname{GL}_n(\mathbb{C}) \times \mathcal{T}}(\{*\}) $ as an $\CR$-algebra. Under the identification with $\mathcal{R}[x_1^{\pm 1},\dots,x_n^{\pm 1}]^{\mathfrak{S}_n}$, $[\Lambda^i\C^n]$ is equal to $e_i(x_1,\dots,x_n)$, the $i^{\mathrm{th}}$ elementary symmetric polynomial and $[(\Lambda^n\C^n)^{*}]$ is equal to $(x_1 \cdot \ldots \cdot x_n)^{-1}$ .
\end{prop}
\begin{proof}
 This follows directly from \cite[\S15.5]{FultonHarris}.
\end{proof}

\cref{deff:cyclotomic Hecke} gives directly that the Jucy-Murphy elements $\Lj_i$ are all invertible. We can thus consider the $\mathcal{R}$-algebra homomorphism: $\mathcal{R}[x_1^{\pm 1},...,x_n^{\pm 1}]^{\mathfrak{S}_n} \to \ZJM$.
This gives the following homomorphism of $\mathcal{R}$-algebras
$$\pi:K_{\GL_n(\C) \times \mathcal{T}}(\{*\}) \twoheadrightarrow \ZJM .$$

 Denote by $\Heck_{n,r}$ the algebra $\Heck_{n,r}(\mathcal{R},\bq)$. We also know that $\Heck_{n,r}$ is a cellular algebra (\cite[Theorem 5.5]{Graham-Lehrer-Cellular}, \cite[Theorem 6.3]{AMR06}) with cell datum indexed by $\mathcal{P}^r_n$. For $\lbb \in \mathcal{P}^r_n$, denote by $S^{\lbb}$ the cell $\Heck_{n,r}$-module indexed by $\lbb$ with Murphy basis  $\{c_t \mid t \in \mathrm{SYT}(\lbb)\}$ \cite[Definition 3.14]{DJM98} where $\mathrm{SYT}(\lbb)$ is the set of standard Young tableaux of shape $\lbb$. Now by construction the Jucys-Murphy's elements $\Lj_1,\dots, \Lj_n$ act triangularly on the basis $\{c_t \mid t \in \mathrm{SYT}(\lbb)\}$:
\begin{equation*}
\Lj_i \cdot c_t=\alpha_t(i)c_t + \sum_{v \triangleright t}{r_{i,t}^vc_v},
\end{equation*}
for some $\alpha_t(i),r_{i,t}^v \in \mathcal{R}$.
\begin{deff} For $t \in\mathrm{SYT}(\lbb)$, let
\begin{equation*}
    \alpha_t:\begin{array}{ccc}
    \ZJM &\to & \mathcal{R} \\
    f & \mapsto & f(\alpha_t(1),\dots,\alpha_t(n))
    \end{array}.
\end{equation*}
\end{deff}

\begin{prop} 
For each $\lbb \in \mathcal{P}^r_n$ and $s,t \in \mathrm{SYT}(\lbb)$, the following equality holds:
\begin{equation*}
\alpha_s = \alpha_t.
\end{equation*}
\end{prop}
\begin{proof}
This is a direct application of \cite[Proposition 3.2]{Li12} and \cref{lemma:incl-center}.
\end{proof}

As $\alpha_t$ does not depend on $t$, set $\alpha_{\lbb}:=\alpha_t$ for any $t$.

\begin{lemme}
    \label{lemma: End ring small}
    The endomorphism ring $\End_{\Heck_{n,r}}(S^{\lbb})$ is isomorphic to $\mathcal{R}$, for each $\lbb \in \mathcal{P}^r_n$.
\end{lemme}
\begin{proof}
    We know that the endomorphism ring of the Specht modules consists of scalar endomorphisms after tensoring with the fraction field of $\CR$, since these are then irreducible \cite[Theorem 3.7]{ArikiKoike}. Moreover, these modules are free over $\mathcal{R}$, hence the desired result since a scalar in the fraction field of $\CR$ preserving a nonzero free $\CR$-module lies in $\CR$.
\end{proof}

For $\lbb \in \mathcal{P}^r_n$, the action of $z \in Z(\Heck_{n,r})$ on $S^{\lbb}$ defines an element of $\End_{\Heck_{n,r}}(S^{\lbb})$. Using \cref{lemma: End ring small}, we have a unique scalar $\alpha_{\lbb}(z)\in \mathcal{R}$. We can now define the following homomorphism:
\begin{equation*}
    \sigma:\begin{array}{ccc}
    Z(\Heck_{n,r}) &\to & \bigoplus_{\lbb \in \mathcal{P}_n^r}{\mathcal{R}} \\
    z & \mapsto & (\alpha_{\lbb}(z))_{\lbb}
    \end{array}.
\end{equation*}

We will abuse notation and denote also by $\sigma$ the restriction to $\ZJM$.
\begin{lemme}
\label{lemma: sigma-inj}
    The homomorphism $\sigma:\ZJM \to  \bigoplus_{\lbb \in \mathcal{P}_n^r}{\mathcal{R}}$ is injective.
\end{lemme}
\begin{proof}
    The algebra $\Heck_{n,r}$ is a free $\mathcal{R}$-module. In particular it is torsion free. So $\ZJM$ is also an $\mathcal{R}$-torsion free module. 
    The fiber of $\ZJM$ at a Zariski generic point of $\operatorname{Spec}\mathcal{R}$ has rank $|\mathcal{P}^r_n|$ \cite[Theorem 3.20]{ArikiKoike} and \cite[Theorem 1.4]{HuShi}. 
    On the other hand, for a Zariski generic point, the specialization of $\sigma$ to this point is surjective 
    because the scalars that arise from the action (image of $\sigma \circ \pi$) coincide with their $K$-theoretic counterpart (see the proof of \cref{thm:main}) and for the $K$-theory we have a surjection by the localisation theorem; a purely algebraic proof follows from the fact that for a Zariski generic point the content function separates multipartitions, see \cite[Theorem 2.19]{Mat04} or \cite[Lemma 2.15]{HuShi}.
    
    This implies that $\sigma$ is generically an isomorphism. In particular, it becomes an isomorphism after tensoring with $\operatorname{Frac}\mathcal{R}$. It follows that the kernel of $\sigma$ must be an $\mathcal{R}$-torsion, hence, zero (since $\ZJM$ is $\mathcal{R}$-torsion free).  
\end{proof}

\begin{prop}
\label{prop:cyclic-struct-alg} 
        Under the map $\sigma \circ \pi$ the image of $[\Lambda^i\C^n]$ in $\bigoplus_{\mathcal{P}_n^r}{\mathcal{R}}$  is the family  of elementary symmetric polynomials
  \begin{equation*}  
         \left (\alpha_{\lbb}(e_i)\right)_{\lbb \in \mathcal{P}_n^r}. 
    \end{equation*}
\end{prop}
\begin{proof}
This is a direct consequence of \cref{prop:generators_in_k-theory}.
\end{proof}

\begin{lemme}
\label{lem:cyclic-struct-alg}
The following equality holds:
\begin{equation*}
    \alpha_{t}(i) = q^{b-a}Q_c
\end{equation*}
where $x=(a,b,c)$ is the unique node $(a,b)$ in the Young diagram of $\lb^c$ such that $t(x)=i$. 
\end{lemme}
\begin{proof}
This follows from \cite[Proposition 3.7]{GM00}.
\end{proof}

\subsection{Equivariant K-theory of Gieseker space}

Recall the following well-known result (see, for example, \cite[Theorem 3.7]{Nakajima-Yoshioka-instanton-counting}). 
\begin{prop}
\label{prop:free-mod}
The algebra $K_{\mathcal{T}}(\Gie(n,r))$ is a free $K_{\mathcal{T}}(\{*\})$-module with basis indexed by the $\mathcal{T}$-fixed points of $\Gie(n,r)$.
\end{prop}
\begin{proof}
 $\Gis$ is a smooth variety over $\BC$ with finitely many fixed points under the action of $\torus\times \BC^{\times}_\hbar$. It is proper over $\mathcal{G}_0(n,r)$, the affine quotient, and the latter is contracted to zero via the action of $\mathbb{C}^\times_\hbar$. We can thus take a generic cocharacter $\nu$ of torus $\mathcal{T}$  and consider a cocharacter $(\nu, t^N)$ for sufficiently large $N$. This action will contract $\Gis$ to $\Gis^{\torus}$. Applying the  Białynicki-Birula Theorem we obtain a $\torus\times \BC^{\times}_\hbar$-equivariant affine paving for $\Gis$ which is an $\alpha$-partition in the sense of \cite[Section 1.3]{de-concini-lusztig-procesi-1988}. 
 Now, using the cellular fibration lemma (see, for example, \cite[Lemma 5.5.1]{CG}), we obtain an explicit basis of  $K_{\mathcal{T}}(\Gie(n,r))$ over $\CR=K_{\mathcal{T}}(\{*\})$ consisting of the structure sheaves of the unions of attractors to fixed points.

\end{proof}

\noindent The morphism $\iota_{\mathcal{G}}: \Gie(n,r)^{\mathcal{T}} \hookrightarrow \Gie(n,r)$  induces a restriction morphism 
\begin{equation*}
    \iota_{\mathcal{G}}^*:K_{\mathcal{T}}(\Gie(n,r)) \to K_{\mathcal{T}}\left(\Gie(n,r)^{\mathcal{T}}\right). 
\end{equation*}

\begin{lemme}
The morphism $\iota_{\mathcal{G}}^*$ is a monomorphism.
\end{lemme}
\begin{proof}
    By Thomason's algebraic localization theorem \cite[Theorem 2.1]{Thom92}, the morphism $\iota_{\mathcal{G}}^*$ is an isomorphism after tensoring it with the fraction field of $\CR$ Combining \cref{prop:free-mod} with the fact that the variety $\Gie(n,r)^{\mathcal{T}}$ is a finite number of points, we obtain the claim.
\end{proof}

Recall that $\mu^{-1}(0)^{\mathrm{st}}$ denotes the semistable points in the fiber of the moment map $\mu$, such that $\Gie(n,r) \simeq \mu^{-1}(0)^{\mathrm{st}} \sslash GL_n(\C)$.
One important source of vector bundles on $\Gie(n,r)$ (more generally on Nakajima quiver varieties) comes from $\GL_n(\C)$-representations.  Indeed, if $V$ is such a representation, then one can consider the contracted product $\mu^{-1}(0)^{\mathfrak{st}} \times_{\GL_n(\C)} V$ since $\mu^{-1}(0)^{\mathrm{st}} \to \Gie(n,r)$ is a $\GL_n(\C)$-principal bundle \cite[Proposition 2.3.2]{NakJAMS}.

\begin{deff}
\label{def:taut}
For $V=\C^n$, denote the corresponding tautological vector bundle on $\Gis$ by $\CV$.
\end{deff}

This construction defines a homomorphism $K_{\GL_n(\C)}(\{*\}) \to K(\Gie(n,r))$ and, taking $\mathcal{T}$-equivariance into account, gives the homomorphism of algebras  $\kappa\colon K_{\GL_n(\C) \times \mathcal{T}}(\{*\}) \to K_{\mathcal{T}}(\Gie(n,r))$.


\begin{prop}
The homomorphism $\kappa$ is an epimorphism.
\end{prop}

\begin{proof}

We need to show that the homomorphism
\begin{equation}\label{eq:homom_restr_open}
 K_{\operatorname{GL}_n(\mathbb{C}) \times \mathcal{T}}(\{*\}) \rightarrow K_{\operatorname{GL}_n(\mathbb{C}) \times \mathcal{T}}(\mu^{-1}(0)^{\mathrm{st}}) 
\end{equation} 
induced by the pullback along $\mu^{-1}(0)^{\mathrm{st}} \subset \mu^{-1}(0) \rightarrow \{*\}$ is surjective. This is a homomorphism of $K_{\operatorname{GL}_n(\mathbb{C}) \times \mathcal{T}}(\{*\})=\mathbb{C}[\BT_n]^{\Sk_n} \otimes \mathbb{C}[\mathcal{T}]$-modules where $\BT_n$ denotes the maximal diagonal torus of $\GL_n(\C)$.
It is enough to show that the map \cref{eq:homom_restr_open} stays surjective after specialisation to any $(g,t) \in \BT_n \times \mathcal{T}$ . By the localization theorem (see, for example, \cite[Theorem 4.3]{Edidin-Graham-algebraic_cycles_K-theory}) combined with \cite[Proposition 5.4, (d)]{Edidin-Graham-algebraic_cycles_K-theory}, this fiber identifies with the fiber at $(1,1)$ of the pullback homomorphism:
\begin{equation}\label{eq:homom_restr_open_after_loc}
K_{Z_{\operatorname{GL}_n(\mathbb{C})}(g) \times \mathcal{T}}(\{*\}) \rightarrow K_{Z_{\operatorname{GL}_n(\mathbb{C})}(g) \times \mathcal{T}}((\mu^{-1}(0)^{\mathrm{st}})^{(g,t)})
\end{equation}
so it remains to prove that the fiber of \cref{eq:homom_restr_open_after_loc} at $(1,1)$ is surjective. Recall that $\cotrepspace$ is the cotangent bundle to the representation space of the framed Jordan quiver \cref{eqn:cotrepspace}. Note now that 
\begin{equation*}
(Z_{\operatorname{GL}_n(\mathbb{C})}(g),\cotrepspace^{(g,t)})
\end{equation*}
comes from some quiver by \cite[Proposition 2.13]{dumanski-krylov_2025} (see also \cite[Section 2]{Pae25}), so the map \cref{eq:homom_restr_open} identifies with the fiber at $(1,1)$ of the analogous pullback homomorphism corresponding to the quiver theory  $(Z_{\operatorname{GL}_n(\mathbb{C})}(g), \cotrepspace^{(g,t)})$. Set $\mathfrak{M}:=\mu_{(Z_{\operatorname{GL}_n(\mathbb{C})}(g),\cotrepspace^{(g,t)})}^{-1}(0)^{\mathrm{st}} \sslash Z_{\operatorname{GL}_n(\mathbb{C})}(g)$, where $\mu_{(Z_{\operatorname{GL}_n(\mathbb{C})}(g),\cotrepspace^{(g,t)})}$ is the corresponding moment map.

It remains to note that the homomorphism 
\begin{equation}\label{eq:MN_surj_nonequiv}
K_{Z_{\operatorname{GL}_n(\mathbb{C})}(g)}(\{*\}) \rightarrow K(\mathfrak{M}) = K_{Z_{\operatorname{GL}_n(\mathbb{C})}(g)}(\mu_{(Z_{\operatorname{GL}_n(\mathbb{C})}(g),\cotrepspace^{(g,t)})}^{-1}(0)^{\mathrm{st}})
\end{equation}
is surjective by \cite[Corollary 1.4]{McN18} (applied to the quiver variety $\mathfrak{M}$)  and taking the fiber of \cref{eq:MN_surj_nonequiv} we conclude the desired surjectivity.
\end{proof}

\begin{rmk}
    For cohomology this follows from \cite[Corollary 1.5]{McN18}. The authors of \cite{McN18} also write that ``Analogues of Corollary 1.5 can also be proven for $K$-theory and elliptic
cohomology equivariant with respect to a torus $\mathbb{T}$ or more general ``flavor symmetries'' of $\mathfrak{M}$''. 
\end{rmk}

To understand $K_{\mathcal{T}}(\Gie(n,r))$, it is thus enough to understand the image of a set of generators of $K_{\operatorname{GL}_n(\mathbb{C}) \times \mathcal{T}}(\{*\})$ over $K_{\mathcal{T}}(\{*\})$ under $ \iota_{\mathcal{G}}^* \circ \kappa$. In light of \cref{prop:generators_in_k-theory} pick $[\Lambda^i V]$, $[(\Lambda^n V)^*]$ as such a set. 

\begin{prop}
\label{prop:cyclic-struct-geom} 
        Under the map $\iota_{\mathcal{G}}^* \circ \kappa$ the image of $[\Lambda^i V]$ in ${\mathcal{R}^{\oplus{|\mathcal{P}_n^r|}}}$  is the family  of elementary symmetric polynomials  
  \begin{equation*}
         \left (e_i(q^{\ct_1^1}Q_1,\ldots,q^{\ct_{|\lb^1|}^1}Q_1,\ldots, q^{\ct_{|\lb^r|}^r}Q_r) \right)_{\lbb \in \mathcal{P}_n^r}
    \end{equation*}
where $\ct^i_j$ denotes the content of the box with number $j$ in $\lb^i$, for a given order of the boxes. The image of $[(\Lambda^n V)^*]$ is the inverse of $\iota_{\mathcal{G}}^*(\kappa([\Lambda^n\C^n]))$.

\end{prop}

\begin{proof}

   Under $\iota_{\mathcal{G}}^* \circ \kappa$  the image of $[\Lambda^i V]$ in the fixed point summand corresponding to $\lbb$ is the restriction $[\Lambda^i\mathcal{V}]_{|\lbb}$ of $[\Lambda^i\mathcal{V}]$ (\cref{def:taut}) at this fixed point. \cref{prop:weights_at_fixed_point} describes $[\mathcal{V}]|_{\lbb}$ as the representation of $\torus$. Indeed, recall that at each fixed point $\lbb$, $\torus$ acts on $\CV|_{\lbb}$ by the contents of Young diagrams in the corresponding multipartition, shifted by $(Q_1, \ldots, Q_r)$.
    Correspondingly, the restrictions $[\Lambda^i\mathcal{V}]_{|\lbb}$ are given by the symmetric polynomials in the $\torus$-characters of $[\mathcal{V}]_{|\lbb}$.
   Their images can thus be written as 
    \begin{equation*}
       \iota_{\mathcal{G}}^* \circ \kappa ([\Lambda^i V])_{|\lbb} =  e_i(q^{\ct_1^1}Q_1,\ldots,q^{\ct_{|\lb^1|}^1}Q_1,\ldots, q^{\ct_{|\lb^r|}^r}Q_r).
    \end{equation*}
   The last assertion follows by the definition of the torus action on the representation dual to $\Lambda^n\C^n$.
\end{proof}

\noindent The global situation is as follows:
\begin{center}
\begin{tikzcd}
\& K_{\GL_n(\C) \times \mathcal{T}}(\{*\}) \ar[dl, "\kappa"', two heads] \ar[dr, "\pi", two heads] \& \\
K_{\mathcal{T}}(\Gie(n,r)) \ar[d,"\iota_{\mathcal{G}}^*"', ->, hook] \&  \& \ZJM \ar[d, "\sigma"', ->, hook] \\
K_{\mathcal{T}}\left(\Gie(n,r)^{\mathcal{T}}\right) \ar[rr, no head, "{\finesim}"{marking, yshift=1ex}]  \&  \& {\mathcal{R}^{\oplus{|\mathcal{P}_n^r|}}}
\end{tikzcd}
\end{center}

\begin{thm}\label{thm:main}
\label{main_thm}
The following $\C[q^{\pm 1}, Q_1^{\pm 1},\dots,Q_r^{\pm 1}]$-algebras are isomorphic:
\begin{equation*}
K_{\mathcal{T}}\left(\Gie(n,r)\right) \iso \ZJM.
\end{equation*}
The isomorphism is given by:
\begin{equation*}
[\Lambda^i\mathcal{V}] \mapsto e_i(\Lj_1,\ldots, \Lj_n),~[{(\Lambda^n \mathcal{V})}^*] \mapsto  \Lj_1^{-1} \cdot \ldots \cdot \Lj_n^{-1}.
\end{equation*}
\end{thm}
\begin{proof}
It is enough to verify the following identities:
\begin{equation*}
    (\iota_{\mathcal{G}}^*\circ\kappa)[\Lambda^iV]=(\sigma\circ\pi)[\Lambda^iV], \qquad \forall i \in \llbracket 1, n \rrbracket
\end{equation*}
and
\begin{equation*}
(\iota_{\mathcal{G}}^*\circ\kappa)([{(\Lambda^nV)}^*])=(\sigma\circ\pi)([({\Lambda^nV)}^*]).
\end{equation*}
These follow from comparing \cref{prop:cyclic-struct-alg}, together with \cref{lem:cyclic-struct-alg}, and \cref{prop:cyclic-struct-geom}. 
\end{proof}

\begin{rmk}\label{rem:basis in JM center}
    \cref{prop:free-mod} says that there exists a geometrically defined basis of $\KTGis$ over $\CR$. It would be interesting to see what it corresponds to on the algebraic side, namely in $\ZJM$.
\end{rmk}

\begin{rmk}
  Note that $K_{\mathcal{T}}(\mathcal{G}(n,r))$ has a one-parameter deformation given by $K_{\mathcal{T} \times \mathbb{C}^\times_\hbar}(\mathcal{G}(n,r))$. We do not know what corresponds to this deformation on the cyclotomic affine Hecke algebra side. 
\end{rmk}



  \section{Pairing and the cocenter}\label{sec:center-cocenter}
  \subsection{Geometric side}
Consider the (affinization) morphism $\mathcal{G}(n,r) \rightarrow \mathcal{G}_0(n,r)$ and let $\mathcal{L}(n,r) \subset \mathcal{G}(n,r)$ be the preimage of $0 \in \mathcal{G}_0(n,r)$. The variety $\mathcal{L}(n,r)$ is proper and (usually) singular. 

\

It follows from \cite[Theorem 7.3.5]{NakJAMS} that there is a perfect pairing:
\begin{equation}\label{eq:geom_pairing}
K_{\mathcal{T}}(\mathcal{G}(n,r)) \times K^{\mathcal{T}}(\mathcal{L}(n,r)) \ni (F,F') \mapsto p_*(F \otimes^L_{\mathcal{G}(n,r)} F')  \in K_{\mathcal{T}}(\{*\}),
\end{equation}
where $p\colon \mathcal{G}(n,r) \rightarrow \{*\}$ is the constant map. We will denote this pairing by $\langle\,-,-\,\rangle$.

Note that $\mathcal{L}(n,r)^{\mathcal{T}} = \mathcal{G}(n,r)^{\mathcal{T}}=\mathcal{P}^r_n$.  Let us denote by
\begin{equation*}
\iota_{\mathcal{L}}\colon \mathcal{L}(n,r)^{\mathcal{T}} \hookrightarrow \mathcal{L}(n,r),~\iota_{\mathcal{G}}\colon \mathcal{G}(n,r)^{\mathcal{T}} \hookrightarrow \mathcal{G}(n,r)
\end{equation*}
the natural embeddings. 
They induce morphisms $\iota_{\mathcal{L},*}:K^{\mathcal{T}}(\mathcal{L}(n,r)^{\mathcal{T}}) \to K^{\mathcal{T}}(\mathcal{L}(n,r))$, $\iota_{\mathcal{G}}^*:K_{\mathcal{T}}(\Gie(n,r)) \to K_{\mathcal{T}}\left(\Gie(n,r)^{\mathcal{T}}\right)$ both of which are $\mathcal{R}$-linear and become isomorphisms after tensoring by $\operatorname{Frac}(\mathcal{R})$, the fraction field of $\CR$. Let $\langle-,-\rangle_{\mathcal{P}^r_n}$ be the standard pairing on $K_{\mathcal{T}}(\mathcal{P}^r_n)$ given by 
\begin{equation*}
\langle [\mathcal{O}_{\lambda^\bullet}],[\mathcal{O}_{\mu^\bullet}]\rangle_{\mathcal{P}^{r}_n} = \delta_{\lambda^\bullet,\mu^{\bullet}}.
\end{equation*}

Note that $K^{\mathcal{T}}(\mathcal{L}(n,r))$ is a {\emph{module}} over the algebra $K_{\mathcal{T}}(\mathcal{G}(n,r))$ with the action given by the (derived) tensor product. The following lemma is standard and ensures the compatibility of the pairings with each other as well as with the module structures.
\begin{lemme}\label{lem: push pull for iota are adjoint}
(a) For $E \in K_{\mathcal{T}}(\mathcal{G}(n,r))$, $F, F' \in K^{\mathcal{T}}(\mathcal{L}(n,r))$, we have $\langle E \otimes^L F,F'\rangle = \langle F,E \otimes^L F'\rangle$.

(b) For $E \in K_{\mathcal{T}}(\mathcal{G}(n,r))$ and $P \in \KTmpart$ we have
$
\langle E, \iota_{\mathcal{L},*}P\rangle = \langle \iota_{\mathcal{G}}^*E,P\rangle_{\mathcal{P}^r_n}.
$
\end{lemme} 
\begin{proof}
Part $(a)$ follows from commutativity and associativity of the tensor product. 

Let us prove part $(b)$.
By projection formula, we have 
\begin{equation*}
E \otimes^L_{\mathcal{G}(n,r)} \iota_{\mathcal{G},*}P = \iota_{\mathcal{G},*}(\iota_{\mathcal{G}}^*(E) \otimes_{\mathcal{P}^r_n} P)
\end{equation*}
and part (b) follows. 
\end{proof}

\subsection{Algebraic side}
From now on, and throughout the remainder of \cref{sec:center-cocenter} we assume that $q \neq 1$.
Let $\operatorname{Tr}({\bf{H}}_{n,r})$ be the {\emph{cocenter}} of the algebra ${\bf{H}}_{n,r}$, i.e. ${\bf{H}}_{n,r}/\!\left [{\bf{H}}_{n,r},{\bf{H}}_{n,r} \right ]$.  
In \cite[Definition 2.4]{HuShi} authors consider the map $\tau_{\mathcal{R}} \colon {\bf{H}}_{n,r} \rightarrow \mathcal{R}$ defined on the basis \cref{eqn:basis-cycl} by 
\begin{equation*}
\tau_{\mathcal{R}}(\Lj_1^{a_1}\dots \Lj_n^{a_n}\Tj_w):=\begin{cases}
1, &\text{if $w=1$ and $a_1=\dots=a_n=0$;}\\
0, &\text{otherwise.}
\end{cases}
\end{equation*}
It follows from  \cite[Proposition 2.2]{MM98} that $\tau_{\mathcal{R}}$ induces the symmetric $\mathcal{R}$-algebra structure on ${\bf{H}}_{n,r}$. In particular, we have a perfect pairing:
\begin{equation}\label{eq:pairing between center and cocenter}
Z({\bf{H}}_{n,r}) \times \operatorname{Tr}({\bf{H}}_{n,r}) \rightarrow \mathcal{R},~(a,b) \mapsto \tau_{\mathcal{R}}(ab).
\end{equation}
Abusing notation, we will denote the pairing (\ref{eq:pairing between center and cocenter}) by $\langle-,-\rangle$.

Recall that we have an $\mathcal{R}$-linear homomorphism of algebras $\sigma\colon Z({\bf{H}}_{n,r}) \rightarrow \bigoplus_{\mathcal{P}^r_n}\mathcal{R}$. Taking the dual and identifying the dual to $Z({\bf{H}}_{n,r})$ with $\operatorname{Tr}({\bf{H}}_{n,r})$ via the pairing (\ref{eq:pairing between center and cocenter}) and the finite freeness of $\operatorname{Tr}({\bf{H}}_{n,r})$ cf. \cite[Theorem 1.4(1)]{HuShi},  we obtain an $\mathcal{R}$-linear map ${\sigma^*:\bigoplus_{\mathcal{P}^r_n}\mathcal{R} \rightarrow \operatorname{Tr}({\bf{H}}_{n,r})}$. 
The following lemma mimics Lemma \ref{lem: push pull for iota are adjoint} above.
\begin{lemme}\label{lem:pairing on algebraic side is compatible with module structure and duality}
    (a) For $a,b \in Z({\bf{H}}_{n,r})$ and $c \in \operatorname{Tr}({\bf{H}}_{n,r})$ we have $\langle ab,c\rangle=\langle b,ac\rangle$.
    
    (b) For $a \in Z({\bf{H}}_{n,r})$ and $x \in \bigoplus_{\mathcal{P}^r_n} \mathcal{R}$ we have $\langle a,\sigma^*(x)\rangle = \langle \sigma(a),x\rangle_{\mathcal{P}^r_n}$.
\end{lemme}
\begin{proof}
Part $(a)$ follows from the equality $\langle ab,c\rangle=\tau_{\mathcal{R}}(abc)=\tau_{\mathcal{R}}(bac)=\langle b,ac\rangle$ (use that $a$ is central). 
Part $(b)$ follows from the definition of $\sigma^*$. 
\end{proof}

We define the action $Z({\bf{H}}_{n,r}) \curvearrowright \bigoplus_{\mathcal{P}^r_n} \mathcal{R}$ via $\sigma$. Namely, $a \cdot x = \sigma(a)x$ for $a \in Z({\bf{H}}_{n,r})$ and $x \in \bigoplus_{\mathcal{P}^r_n} \mathcal{R}$.    
\begin{lemme}\label{lem:sigma dual is a homomorphism}
The map $\sigma^*$ is a homomorphism of $Z({\bf{H}}_{n,r})$-modules.
\end{lemme}
\begin{proof}
It is enough to show that $\langle \sigma^*(a \cdot x),b\rangle=\langle a\sigma^*(x),b\rangle$ for any $a,b \in Z({\bf{H}}_{n,r})$ and ${x \in \bigoplus_{\mathcal{P}^r_n} \mathcal{R}}$. Using Lemma \ref{lem:pairing on algebraic side is compatible with module structure and duality} together with the fact that $\sigma$ is a homomorphism  we get:
\begin{multline*}
\langle b, \sigma^*(a \cdot x)\rangle =
\langle\sigma(b), a \cdot x \rangle_{\mathcal{P}^r_n}=
\langle \sigma(b), \sigma(a)x \rangle_{\mathcal{P}^r_n} = 
\\ 
=\langle \sigma(a)\sigma(b), x \rangle_{\mathcal{P}^r_n}=
\langle \sigma(ab),x \rangle_{\mathcal{P}^r_n}=
\langle ab, \sigma^*(x)\rangle =
\langle b, a\sigma^*(x)\rangle
\end{multline*}
as desired.
\end{proof}

\subsection{Main conjecture}\label{subsec:main conjecture} From now on we will use the following notation. For an $\mathcal{R}$-module $P$ we will denote by $P_{\mathrm{loc}}$ the tensor product $P \otimes_{\mathcal{R}} \operatorname{Frac}\mathcal{R}$. 
\begin{lemme} \label{lemma: unique-iso-loc}  There exists a unique isomorphism of $Z({\bf{H}}_{n,r})_{\mathrm{loc}}$-modules:
\begin{equation}\label{eq:iso cocenter K theory of L after loc}
K^{\mathcal{T}}(\mathcal{L}(n,r))_{\mathrm{loc}} \simeq \operatorname{Tr}({\bf{H}}_{n,r})_{\mathrm{loc}}
\end{equation}
that intertwines pairings (\ref{eq:geom_pairing}), (\ref{eq:pairing between center and cocenter}) as well as $\iota_{\mathcal{L},*}$ with $\sigma^*$.    
\end{lemme}
\begin{proof}
Clearly, there is a unique identification of $\operatorname{Frac}\mathcal{R}$-modules  $K^{\mathcal{T}}(\mathcal{L}(n,r))_{\mathrm{loc}} \simeq \operatorname{Tr}({\bf{H}}_{n,r})_{\mathrm{loc}}$ intertwining the pairings (\ref{eq:geom_pairing}), (\ref{eq:pairing between center and cocenter}). 
We already know that $\iota_{\mathcal{G}}^*$ identifies with $\sigma$ so  Lemmas \ref{lem: push pull for iota are adjoint}, \ref{lem:pairing on algebraic side is compatible with module structure and duality} imply that $\iota_{\mathcal{L},*}$ identifies with $\sigma^*$. The compatibility with module structures follows from \cref{lem:sigma dual is a homomorphism}. 
\end{proof}

We make the following conjecture. 
\begin{conj}\label{conj: iso cocenter and K theory of L}
The isomorphism (\ref{eq:iso cocenter K theory of L after loc}) induces an isomorphism 
\begin{equation*}
K^{\mathcal{T}}(\mathcal{L}(n,r))_{q \neq 1} \simeq \operatorname{Tr}(\Heck_{n,r})_{q \neq 1}
\end{equation*}
of $Z(\Heck_{n,r})_{q \neq 1}$-modules. 
\end{conj}

\cref{prop:our prop claiming that if conj holds then JM is equal to center} below is an immediate corollary of Conjecture \ref{conj: iso cocenter and K theory of L}. 
When $r=1$ this is \cite[Theorem 7.2]{FG06}, when $n=2$, this is stated as \cite[Theorem 2.5]{McG12} and when the algebra is semisimple see \cite[Theorem 3.20]{ArikiKoike}. \cref{conj: iso cocenter and K theory of L}   is proven by Hu, Mathas, Shi, and Webster in \cite[Theorem C(2)]{HMSW26}. They reversed the logic of this section by first proving that \cref{prop:our prop claiming that if conj holds then JM is equal to center} is true in general and over any commutative ring.


\

We know that $\ZJM \subset Z(\Heck_{n,r})$ (cf. \cref{lemma:incl-center}).
It follows from the construction of the isomorphisms in \cref{thm:main} and \cref{conj: iso cocenter and K theory of L} that the following diagram is commutative.
\begin{center}
\begin{tikzcd}[column sep=large]
\& \ZJM  \ar[dl, ->,hook] \ar[r, -, "\sim"', "\cref{thm:main}"] \& K_{\mathcal{T}}(\Gie(n,r))  \ar[d, <->, "{\rotatebox[origin=c]{90}{$\sim$}}"] \\
Z(\Heck_{n,r}) \ar[r, <->, "\sim"] \& {\operatorname{Tr}(\Heck_{n,r})}^*  \ar[r, -, "\sim"', "\cref{conj: iso cocenter and K theory of L}"]\& {K^{\mathcal{T}}(\mathcal{L}(n,r))}^* 
\end{tikzcd}.
\end{center}

\begin{prop}\label{prop:our prop claiming that if conj holds then JM is equal to center} 
Assuming \cref{conj: iso cocenter and K theory of L}, the Jucys-Murphy center $\ZJM_{q \neq 1} $ is equal to the center $Z(\Heck_{n,r})_{q \neq 1}$.  
\end{prop}

We finish this section with the following lemma. 
\begin{lemme}
Isomorphism (\ref{eq:iso cocenter K theory of L after loc}) induces an embedding $\operatorname{Tr}(\Heck_{n,r}) \hookrightarrow K^{\mathcal{T}}(\mathcal{L}(n,r))$.   
\end{lemme}
\begin{proof}
Pick an element $x \in \operatorname{Tr}(\Heck_{n,r})$. Our goal is to check that its image under (\ref{eq:iso cocenter K theory of L after loc}) lives in $K^{\mathcal{T}}(\mathcal{L}(n,r))$. 
Since $K^{\mathcal{T}}(\mathcal{L}(n,r))$ is dual of $K_{\mathcal{T}}(\mathcal{G}(n,r)) \simeq \ZJM$ and (\ref{eq:iso cocenter K theory of L after loc}) is compatible with the pairings it is enough to show that the pairings of $x$ with elements of $\ZJM$ lie in $\mathcal{R}$. This follows from the definition of the symmetrizing form $\tau_{\mathcal{R}}$. 
\end{proof}



  \section{Applications}\label{sec:Applications}




\subsection{Hilbert scheme of points} 
For $r=1$ it follows from \cite[Theorem 2.1, Proposition 2.10]{NakajimaBook} that $\mathcal{G}(n,1) \simeq \operatorname{Hilb}_n(\mathbb{A}^2)$. The algebra ${\bf{H}}_{n,1}(R,{\boldsymbol{q}})$ is  isomorphic to ${\bf{H}}_n(q) \otimes \mathbb{C}[Q_1^{\pm 1}]$, where ${\bf{H}}_n(q)$ is the type $A$ Hecke algebra. 
The Jucys-Murphy’s center of ${\bf{H}}_n(q)\otimes \mathbb{C}[Q_1^{\pm 1}]$ is the symmetric algebra in the elements $\Lj_i$ that are inductively defined by the equations:
\begin{equation*}
\Lj_1=Q_1,~q\Lj_{i+1} = \Tj_i\Lj_i\Tj_i.
\end{equation*}
It follows from \cite[Theorem 7.2]{FG06} that for $q \neq 1$, we have an equality 
\begin{equation*}
Z({\bf{H}}_n(q))^{\mathrm{JM}}_{q \neq 1} = Z({\bf{H}}_n(q))_{q \neq 1}.
\end{equation*}
So, we see that Theorem \ref{thm:main} implies the following description of the equivariant $K$-theory of the Hilbert scheme when specializing $Q_1$ to $1$.


\begin{cor}\label{cor:Hilb case}
There exists an isomorphism of algebras $K_{\mathbb{T}}(\operatorname{Hilb}_n(\mathbb{A}^2)) \simeq Z({\bf{H}}_n(q))^{\mathrm{JM}}$ explicitly given by: 

\begin{equation*}
[\Lambda^i\mathcal{V}] \mapsto e_i(\Lj_1,\ldots, \Lj_n),~[{(\Lambda^n\mathcal{V})}^*] \mapsto  \Lj_1^{-1} \cdot \ldots \cdot \Lj_n^{-1}.
\end{equation*}
After inverting $q-1$, it induces an isomorphism  $K_{\mathbb{T}\\}(\operatorname{Hilb}_n(\mathbb{A}^2))_{q \neq 1} \simeq Z({\bf{H}}_n(q))_{q \neq 1}$.
\end{cor}

\subsection{Specialization at one}
One can specialize \cref{main_thm} to $q~=~1$, recovering the isomorphism of algebras 
\begin{equation*}
\ZJM|_{q=1} \simeq K_{\mathbb{T}_r}(\mathcal{G}(n,r)).
\end{equation*}
Note that $\ZJM|_{q=1}$ is {\emph{not}} isomorphic to the JM-center of ${\bf{H}}_{n,r}|_{q=1}$. Indeed, ${\bf{H}}_{n,r}|_{q=1}$ is simply the smash product $\mathbb{C}\Sk_n \# P_{\boldsymbol{Q}}$ where  
\begin{equation*}
P_{\boldsymbol{Q}} = \mathbb{C}[L_1, \ldots, L_n]/\Big(\prod_{i=1}^r(L_1-Q_i),\ldots,\prod_{i=1}^r(L_k-Q_i),\ldots,\prod_{i=1}^r(L_{n}-Q_i)\Big)
\end{equation*}
and $\Sk_n$ acts on $P_{\boldsymbol{Q}}$ by permuting the variables $L_1, \ldots, L_n$.
Its JM-center is equal to $P_{\boldsymbol{Q}}^{\Sk_n}$ which has rank less than $|\mathcal{P}^r_n|$ in general (see, for example, \cite{McG12}). Note that $\ZJM|_{q=1}$ is in general {\emph{not}} a subalgebra of the smash product $\mathbb{C}\Sk_n \# P_{\boldsymbol{Q}}$. It would be interesting to describe the specialization $\ZJM|_{q=1}$ explicitly.  

For example, if we further specialize to 
\begin{equation*}
{\boldsymbol{Q}}=(Q_1,\ldots,Q_r)=(a_1,\ldots,a_r)=:{\boldsymbol{a}}
\end{equation*} 
for {\emph{distinct}} $a_{k} \in \mathbb{C}^\times$, $k=1,\ldots,r$,
then we get
\begin{multline*}
\ZJM|_{q=1;{\boldsymbol{Q}}={\boldsymbol{a}}} \simeq K(\mathcal{G}(n,r)^{\boldsymbol{a}}) \simeq\bigoplus_{n=n_1+\ldots+n_r} \bigotimes_{k=1}^rK(\operatorname{Hilb}_{n_k}(\mathbb{A}^2)) \simeq \\ 
\simeq \bigoplus_{n=n_1+\ldots+n_r} \bigotimes_{k=1}^r H^*(\operatorname{Hilb}_{n_k}(\mathbb{A}^2)) \simeq \bigoplus_{n=n_1+\ldots+n_r} \bigotimes_{k=1}^r \operatorname{gr}Z(\mathbb{C}\Sk_{n_k}),
\end{multline*}
where $Z(\mathbb{C}\Sk_{n_k})$ is the center of $\mathbb{C}\Sk_{n_k}$ and $\operatorname{gr}$ is taken with respect to the filtration given by $\operatorname{deg}\sigma := n_k-\ell(\sigma)$ (here $\sigma \in \Sk_{n_k}$ and $\ell(\sigma)$ is the number of cycles in the decomposition of $\sigma$ into the product of disjoint cycles). In this chain of the isomorphisms the first is Theorem \ref{thm:main} combined with the localization theorem, the second follows from the explicit description of the fixed points $\mathcal{G}(n,r)^{\boldsymbol{a}}$ (see \cref{eqn:gieseker_fixed_points_into_hilb}, also \cite[Page 18]{Nakajima-Yoshioka-instanton-counting}), the third is the Chern character and the last follows from \cite{VascohomHilb} (see also \cite[Sections 4, 5]{KS25}, \cite[Corollary 4.8]{SVV}).

\subsection{Specialization at roots of unity}
In this section, we will specialize the parameters $\boldsymbol{q}$ to deepen the combinatorial correspondences in \cite{Pae26}.

\vspace{0.3cm}

To keep the notation  consistent with \cite{Pae26}, let us take $\ell \in \Z_{> 1}$ and a multicharge $\ssf \in (\Z/\ell\Z)^r$.
In this subsection, we will specialize $q$ to $\zeta_{\ell}$, a primitive $\ell$-root of unity, and $Q_i$ to $q^{\ssf_i}$. Let us denote by $\Heck_{n,r}(\ell,\ssf)$ this finite-dimensional $\C$-algebra. We want to show that the specialization of  \cref{main_thm} to this case gives a better understanding of the main results of \cite{Pae26}.  Recall that blocks of a finite-dimensional   $\C$-algebra are its indecomposable direct summands as two-sided ideals. Since $\Heck_{n,r}(\ell,\ssf)$ is an Artinian ring, its decomposition into blocks is unique up to permutation. This gives the following decomposition:
\begin{equation}
\label{block decp}
    \Heck_{n,r}(\ell,\ssf) = \bigoplus_{i=1}^k{B_i}.
\end{equation}

We first record the following lemma.

\begin{lemme}
    The following equality holds:
    \begin{equation*}
         Z(\Heck_{n,r}(\ell,\ssf)) = \bigoplus_{i=1}^k{Z(B_i)}.
    \end{equation*}
\end{lemme}
\begin{proof}
Equality \cref{block decp} comes from the existence (and uniqueness, up to permutation) of a family of primitive central idempotents $(\epsilon_i)_{i \in \llbracket 1, k \rrbracket}$, such that $1=\sum_{i=1}^k{\epsilon_i}$. Then $B_i:= \Heck_{n,r}(\ell,\ssf)\epsilon_i$. Since the idempotents are central, it follows that $Z(B_i)=Z(\Heck_{n,r}(\ell,\ssf))\epsilon_i$, hence the desired equality.
\end{proof}

By \cite[Theorem 4.3]{Edidin-Graham-algebraic_cycles_K-theory}, after specialization, the left hand side of \cref{main_thm} is isomorphic to $K(\Gie(n,r)^{\Gamma_{\ssf}})$, where $\Gamma_{\ssf}$ is the cyclic subgroup of $\mathcal{T}$, generated by $(\mathrm{diag}(\zeta_{\ell},\zeta_{\ell}^{\smin 1}),\mathrm{diag}(\zeta_{\ell}^{\ssf_1}, \dots, \zeta_{\ell}^{\ssf r}))$.
Recall from \cite{Pae26} (and using the notation therein) that the indexing set of the connected components of $\Gie(n,r)^{\Gamma_{\ssf}}$ is denoted by 
\begin{equation*}
\mathcal{A}_{\Gamma,\ssf}^n:=\left\{d \in Q^+_{\Gamma} \mid \sum_{i=0}^{\ell\smin 1}{d_i}=n, \Lambda^{\ssf} - d \in P(\Lambda^{\ssf})\right\}
\end{equation*}
and these components are isomorphic to affine type $A$ quiver varieties.

The following proposition refines the combinatorial bijection observed in \cite{Pae26}.
\begin{prop}\label{prop:raphael_block_upgrade_at_q_root_of_unity}
    The $\C$-algebra $K\left(\Gie(n,r)^{\Gamma_{\ssf}}\right)$ admits the following block decomposition:
    \begin{equation*}
        K\left(\Gie(n,r)^{\Gamma_{\ssf}}\right) \simeq \bigoplus_{d \in \mathcal{A}_{\Gamma,s}}{K\left(\QV_d^{\Gamma_{\ssf}}\right)}.
    \end{equation*}
\end{prop}
\begin{proof}
If $X$ is a smooth algebraic variety with connected components $X_i$, then the structure sheaf of $X$ decomposes:
\begin{equation*}
    \mathcal{O}_X = \prod_i{\mathcal{O}_{X_i}}.
\end{equation*}
Indeed, every $X_i$ is open in $X$ in view of \cite[Tags 04ME and 04MF]{stacks}. At the level of $K$-theory, we obtain a decomposition of $[\mathcal{O}_X]\in K(X)$ as a sum of primitive idempotents, giving the block decomposition of $K(X)$ (here we use that each $K(X_i)$ is local, for example, because it identifies with $H^*(X_i)$ via the Chern character isomorphism). We can now conclude thanks to \cite[Theorem 2.15]{Pae26}, which gives the decomposition into connected components of $\Gie(n,r)^{\Gamma_{\ssf}}$. 
\end{proof}


In particular, we have the following result identifying $K$-theories of affine type $A$ quiver varieties with centers of the corresponding blocks of $H_{n,r}(\ell,\ssf)$.
\begin{thm}
    \label{thm:blocks}
     Take $d \in \mathcal{A}^n_{\Gamma,\ssf}$. Assuming \cref{conj: iso cocenter and K theory of L}, we have a natural isomorphism of $\C$-algebras:
    \begin{equation*}
        K\left(\QV_d^{\Gamma_{\ssf}}\right) \simeq Z(B_d).
    \end{equation*}
\end{thm}
\begin{proof}
    This follows from the uniqueness of the block decomposition up to permutation,  \cref{main_thm} and \cite[Theorem 1.4(2)]{HuShi}. Indeed, after specializing to roots of unity, \cref{main_thm} combined with  \cite[Theorem 1.4(2)]{HuShi} gives an explicit identification of $K\left(\QV_d^{\Gamma_{\ssf}}\right)$ with $Z(B_d)$ (under the assumption that \cref{conj: iso cocenter and K theory of L} holds).
\end{proof}



  \bibliographystyle{alpha}
  \bibliography{biblo}

\end{document}